\documentclass[a4paper]{article}

\usepackage{amsthm, amssymb, amsfonts, amsmath,  titlesec, layout, color, graphicx, hyperref}

\usepackage[left=1 in, right = 1 in, includefoot]{geometry}

\titleformat{\section}	{\vspace{.3 in}\centering\large\sc}{\thesection.}{1.5em}{}

\newtheorem{theorem}{Theorem}
\newtheorem{lemma}{Lemma}

\theoremstyle{definition}

\theoremstyle{remark}

\newtheorem{remark}{Remark}

\numberwithin{equation}{section}
\numberwithin{problem}{section}
\numberwithin{corollary}{section}
\numberwithin{theorem}{section}

\newcommand{\A}{\alpha}
\newcommand{\B}{\beta}
\newcommand{\D}{\delta}

\newcommand{\N}{\mathbb{N}}

\newcommand{\Z}{\mathbb{Z}}

\begin{document}

\title{\bf REMARKS ON GENERAL FIBONACCI NUMBER}
\author{\it Masum Billal}

\maketitle

\begin{abstract}  
We dedicate this paper to investigate the most generalized form of {\it Fibonacci Sequence}, one of the most studied sections of the mathematical literature. One can notice that, we have discussed even a more general form of the conventional one. Although it seems the topic in the first section has already been covered before, but we present a different proof here. Later I found out that, the auxiliary theorem used in the first section was proven and even generalized further by {\it F. T. Howard}\eqref{hwd}. Thanks to {\it Curtis Cooper}\eqref{coopr} for pointing out the fact that this has already been studied and providing me with references. For further studies on the literature, one can study \eqref{koshy} and \eqref{levesque}. the At first, we prove that, only the common general Fibonacci Sequence can be a divisible sequence under some restrictions. In the latter part, we find some properties of the sequence, prove that there are infinite {\it alternating bisquable Fibonacci sequence}(defined later) and provide a lower bound on the number of divisors of Fibonacci numbers. 
\end{abstract}



\hrulefill

\section{Introduction}
The {\it General Fibonacci Number} $G_n$ is defined as:
\[G_n=
\begin{cases}
u\mbox{ if }n=0\\
v\mbox{ if }n=1\\
aG_{n-1}+bG_{n-2}\mbox{ if }n\geq2
\end{cases}
\]
Let's denote such a sequence using the notation $\{G\}=\{G_n:(u,v|a,b)\}_{n\in\N}$. Then the common general Fibonacci Number $F_n$ is defined as $\{F\}=\{G_n:(0,1|a,b)\}_{n\in\N}$ i.e. $F_0=0,F_1=1,F_n=aF_{n-1}+bF_{n-2}$. Throughout the whole paper, we take $\{F\}$ and $\{G\}$ for the conventional and most General Fibonacci sequences respectively. We call $\{G\}$ to be a {\it co-prime most general Fibonacci sequence}({\it C-quence} for brevity) for $\gcd(u,v)=\gcd(u,b)=\gcd(a,b)=\gcd(b,v)=1$ if $b\neq0$.

\section{Divisible General Fibonacci Sequences}
A sequence $\{a_n\}_{n\in\N}$ is called {\bf divisibility sequence} if $a_n|a_m$ whenever\footnote{Here $a|b$ denotes $a$ divides $b$.} $n|m$. The main result of this section is:

\begin{theorem}
The only divisible \it{C-quence} are $\{G\}=\{G_n:(0,1|a,b)\}=\{F\}$ and $\{G\}=\{G_n:(u,uk|a,0)\}$.
\end{theorem}

The proof is based on the following auxiliary theorem:

\begin{theorem}
\[G_{m+n+1}=G_{m+1}F_{n+1}+bG_mF_n\]\label{genfib}
\end{theorem}
Other than induction, we can of-course check that this is true using the generalization of {\it Binet's Formula} for the Fibonacci sequence.

\begin{theorem}[Generalized {\it Binet's formula}]
\begin{eqnarray}
G_n & = & v\dfrac{\A^n-\B^n}{\A-\B}+u\dfrac{\A^n\B-\A\B^n}{\A-\B}
\end{eqnarray}
where $\A=\dfrac{-a+\sqrt{\Delta}}{2},\B=\dfrac{-a-\sqrt{\Delta}}{2},\D=a^2+4b$ and $\D\neq0$ i.e. $\A\neq\B$.
and 
\begin{eqnarray}
G_n & = & \left(vn+u\A(1-n)\right)\A^{n-1}
\end{eqnarray} if $\A=\B$.\label{binf}
\end{theorem}

\begin{proof}
Assume that, $G_n=\lambda^n$ for some suitable $\lambda$. Then, we have $\lambda^n=a\lambda^{n-1}+b\lambda^{n-2}$ or $\lambda^2=a\lambda+b$ or\footnote{this can also be called the characteristic equation} \[\lambda^2-a\lambda-b=0\]
which has discriminant $\D=a^2+4\neq0b$ and two roots 
\[\A=\dfrac{a+\sqrt{\Delta}}{2},\B=\dfrac{a+\sqrt{\Delta}}{2}\]
Then $G_n=l\A^{n}+m\B^{n}$ for some integer $l,m$ if $\D\neq0$. We can solve this setting $n=0$ and $n=1$. If $n=0$, $l+m=u$ and if $n=1$, then $l\A+m\B=v$.

In the case $\D=0$, we can assume $G_n=(l+mn)\A^n$ and then we find the second portion to be true setting $n=0,1$. 
\end{proof}

\begin{remark}
We don't need the exact values of $l,m$ to find that the theorem holds true.
\end{remark}

Here is a combinatorial proof of the theorem, generalized idea of tiling.
\begin{proof}[\sc Proof of the auxiliary theorem]
Consider the following tiling problem. There is a $(n+2)\times 1$ rectangle, which has $(n+2)$ squares of size $1\times1$. The first square is the starting square $S$. Then follows the squares $0,1,2,...,n$ totaling $n+1$ squares. The square $S$ along with square $0$ can be painted with $u$ colors and the square $S$ along with square $1$ can be painted with $v$ colors. The rectangle is to be filled with tiles of two types: $1\times1$({\bf type $1$}) and $2\times1$({\bf type $2$}). Type $1$ tile can assume $a$ colors and type $2$ can assume $b$ colors. We can see that, the number of different tiling is $G_n=aG_{n-1}+bG_{n-2}$. And if we consider the tiling of the rectangle starting from square $m$ to square $m+n$ for some integer $m\geq0$, then the number of coloring is $F_{n+1}$(since there is no starting square now).

Now, consider the case where we want to tile a $(m+n+1)$-th square starting from square $S$. There are two cases:

\begin{enumerate}
\item Case $1$: We have to reach the $m+1$-th square to tile, which can be done in $G_{m+1}$ ways. Then we have to tile squares $m+1$ to $m+n+1$, which can be done in $F_{n+1}$ ways.
\item Case $2$: We want to bypass the $m+1$-th square. So, we tile upto $m$-th square, which can be done in $G_m$ ways. Then we use the $2\times1$ tiles, which can take $b$ colors and we reach the square $m+2$. Then we can tile from square $m+2$ to $m+n+1$ in $F_n$ ways.
\end{enumerate}
Combining the results of the two cases, we get the total number of coloring is the sum of $G_{m+1}F_{n+1}$(first case) and $bG_mF_n$(second case). On the other hand, we could just color it in $G_{m+n+1}$ ways. Thus, $\boxed{G_{m+n+1}=G_{m+1}F_{n+1}+bG_mF_n}$.

\end{proof}

We will prove some lemmas to prove the main theorem.

\begin{lemma}
If $\{G\}=\{G_n:(u,v|a,b)\}$ is a \it{C-quence}, then $\gcd(b,G_n)=1$ for $n\geq0$.\label{bcp}
\end{lemma}

\begin{proof}
Euclidean algorithm can be used to prove this easily with induction. The base cases $n=0$ and $n=1$ are straight from definition. Now, say that, $\gcd(G_n,b)=1$. We find that, 
\[\gcd(G_{n+1},b)=\gcd(aG_n+bG_{n-1},b)=\gcd(b,aG_n)=\gcd(b,G_n)=1\]
which completes the inductive step.

\end{proof}

\begin{lemma}
If $\{G\}=\{G_n:(u,v|a,b)\}$ is a \it{C-quence}, then $\gcd(G_{n+1},G_n)=1$ for $n\geq0$.\label{fibcp}
\end{lemma}

\begin{proof}
The base case $n=0$ is trivial from the definition again. And we assume that, $\gcd(G_{n+1},G_n)=1$. To complete the inductive step we can show
\begin{eqnarray*}
\gcd(G_{n+2},G_{n+1}) & = & \gcd(aG_{n+1}+bG_n,G_{n+1})\\
					  & = & \gcd(bG_n,G_{n+1})\\
					  & = & \gcd(G_n,G_{n+1})\mbox{ using lemma \eqref{bcp}}\\
					  & = & 1
\end{eqnarray*}
\end{proof}

\begin{lemma}
$\{F\}$ is a divisible sequence i.e. $F_n|F_{nk}$ for all $n,k\geq0$.\label{fibdiv}
\end{lemma}

\begin{proof}
Setting $\{F\}=\{G\}$ and $n+1=mq$ in equation \eqref{genfib},
\begin{equation}
F_{m(q+1)}=F_{m+1}F_{mq}+bF_mF_{mq-1}\label{fibd}
\end{equation}
Now, we can easily induct on $q$. The case $q=1$ is clear. Then if we take $F_m|F_{mq}$, from equation \eqref{fibd}, 
\[F_m|F_{m+1}F_{mq}+bF_mF_{mq-1}=F_{m(q+1)}\]
Thus, the induction step is complete.
\end{proof}

\begin{lemma}
$\gcd(F_m,F_n)=F_{\gcd(m,n)}$.\label{fibgcd}
\end{lemma}

\begin{proof}
We already know that, \[F_{m+n+1}=F_{m+1}F_{n+1}+bF_mF_n\]
Set $(n,m)\rightarrow(mq-1,r)$:
\[F_{mq+r}=F_{r+1}F_{mq}+bF_rF_{mq-1}\]
Therefore, using Euclidean algorithm, if $n=mq+r$,
\begin{eqnarray*}
\gcd(F_n,F_m) & = & \gcd(F_{r+1}F_{mq}+bF_rF_{mq-1},F_m)\\
			  & = & \gcd(F_m,bF_rF_{mq-1})\mbox{ since }F_m|F_{mq}\\
			  & = & \gcd(F_m,F_r)\mbox{ since }\gcd(F_m,b)=\gcd(F_m,F_{mq-1})=1
\end{eqnarray*}
Then, repeating this we reach $\gcd(m,n)$ in the index. Hence, proven.

\end{proof}

Now we prove the key lemma to the theorem.

\begin{lemma}
$\{G\}$ is a divisible \it{C-quence} with $b\neq0$ if and only if $G_m|F_m$.\label{divc}
\end{lemma}

\begin{proof}
Set $n+1=mq$ in the equation \eqref{genfib}. We get
\[G_{m(q+1)}=G_{m+1}F_{mq}+bG_mF_{mq-1}\]\label{diveqn}
First, we prove the if part. If $G_m$ is a divisible sequence, then $G_m|G_{mq}$ for all $q\geq1$. Therefore, $G_m$ divides
$G_{m(q+1)}-bG_mF_{mq-}=G_{m+1}F_{mq}$. But since $\gcd(G_{m+1},G_m)=1$ from lemma \eqref{fibcp}, we infer $G_m|F_{mq}$ for all $q\geq1$. Setting $q=1$, we get $G_m|F_m$.

For the only if part, let's assume $G_m|F_m$. Since $F_m|F_{mq}$, $G_m|F_{mq}$ as well. Then from \eqref{diveqn} we have 
\[G_m|G_{m+1}F_{mq}+bG_mF_{mq-1}=G_{m(q+1)}\]
The induction shows the claim is true.

\end{proof}

\begin{lemma}
If $\{G\}$ is a divisible \it{C-quence} with $b\neq0$, then $\gcd(G_m,F_{mq-1})=1$.\label{ccop}
\end{lemma}

\begin{proof}
Let $d=\gcd(G_m,F_{mq-1})$. Naturally $d|G_m|F_m$ and $d|F_{mq-1}$. Then $d|\gcd(F_m,F_{mq-1})=1$ from lemma \eqref{fibgcd} forcing $d=1$.
\end{proof}

\begin{proof}[\bf Proof Of The Main Theorem]
First we see the case $b=0$. It easily shows that $\{G\}$ can be a divisible sequence if $u|v$. So we can safely assume $b\neq0$.

Since $G_m|F_m$, we easily get that $G_1|F_1=1$ implying $G_1=v=1$. The rest is to prove $u=0$.

Set $(m,n)\rightarrow(0,n-1)$ in equation \eqref{genfib}. We find that, 
\[G_n=vF_n+buF_{n-1}\]
Now, if $\{G\}$ is divisible, then by lemma \eqref{divc}, $F_n=G_nk$ for some integer $k$. Therefore, $G_n=G_nk+buF_{n-1}$ or
\[G_n(1-k)=ubF_{n-1}\]
This shows us that $G_n|ubF_{n-1}$. From lemma \eqref{ccop}, $\gcd(G_n,F_{n-1})=1$ and therefore, $G_n|bu$ for all $n$. If none of $b,u$ is zero, it is impossible to hold. Thus, $bu=0$ and then $u=0$ since $b\neq0$. Therefore $\{G_n:(0,1|a,b)\}$ is the only divisible Fibonacci sequence.

\end{proof}

\section{Bisquare General Fibonacci Numbers}

For $G_0=u,G_1=v$, let's say that $\{G\}_{n\in\N}$ {\bf starts with} the pair $(u,v)$. We call the sum of two squares a {\bf bisquare}. An integer sequence $(a_n)_{n\in\N}$ is called {\bf alternating bisquable sequence} if $a_n$ is a bisquare for all odd $n$ or for all even $n$. If it happens for all odd $n$, then let's call it {\bf oddly bisquable sequence}, otherwise {\bf evenly bisquable sequence}. Denote the number of divisors of $n$ by $\tau(n)$. To establish basic identities we will need the auxiliary matrices: 
\[\mathbf{G}_n
=\begin{pmatrix}
G_{n+2} & G_{n+1}\\
G_{n+1} & G_n
\end{pmatrix},\mathbf{M}=\begin{pmatrix}
a & b\\1&0
\end{pmatrix}
\]

From \eqref{r1}, we can see the proof of {\it Euler} which shows the following claim is true:

\begin{theorem}[Euler]
If $n$ is a bisquare, then so is every divisor of $n$.
\label{bisqr}
\end{theorem}

\begin{theorem}
\[G_nG_{n+2}-G_{n+1}^2=(-b)^n(u^2+uv-v^2)\]
\end{theorem}

\begin{proof}
Notice that, $\mathbf G_n=\mathbf{MG}_{n-1}$ which gives 

\begin{equation}
\mathbf G_n =\mathbf{M}^n\mathbf{G}_0
\label{eq1}
\end{equation}
 We already know that, for two multiplicable matrix $A,B$, $\det(AB)=\det(A)\det(B)$. As a corollary, we also have, $\det(A^m)=\det(A)^m$.  Applying this to equation \eqref{eq1}, we get:

\begin{eqnarray*}
\det(\mathbf G_n)   & = & \det(\mathbf{M}^n\mathbf{G}_0)\\
					& = & \det(\mathbf{M}^n\mathbf{G})\\
G_nG_{n+2}-G_{n+1}^2& = & (-b)^n(G_0G_2-G_1^2)\\
G_nG_{n+2}-G_{n+1}^2& = & (-b)^n(u^2+uv-v^2)
\end{eqnarray*}

\end{proof}

Before going into the proof of main theorem, we find all integer solutions to the equation: $5x^2+4y^2=z^2$. We will see later how this comes into play.

\begin{theorem}
All solutions to the equation $5x^2+4y^2=z^2$ are given by:
\[
(x,y,z)=
\begin{cases}
\left(klm,k\left(\dfrac{5l^2-m^2}{4}\right),k\left(\dfrac{5l^2+m^2}{2}\right)\right)\\
\left(mnk,k\left(\dfrac{l^2-5m^2}{4}\right),k\left(\dfrac{l^2+5m^2}{2}\right)\right)\\
\left(4lmk,k(5l^2-m^2),2k(5l^2+m^2)\right)\\
\left(4lmk,k(l^2-5m^2),2k(l^2+5m^2)\right)\\
\end{cases}
\]
where $l,m\equiv1\pmod2,(l,m)=1$ and $k\in\Z$.
\label{eq2} 
\end{theorem}

Write $\gcd(a,b)$ as $(a,b)$.

\begin{lemma}
If $a$ and $b$ are of the same parity, then $(a+b,a-b)=2(a,b)$, otherwise $(a+b,a-b)=(a,b)$.
\label{cop}
\end{lemma}

\begin{proof}
The proof is merely obvious due to {\it Euclidean Algorithm}.
\end{proof}

\begin{lemma}
If $b^2=ac$, then $a=gl^2,c=gm^2,b=glm$ for some $g,l,m$ with $(l,m)=1$.
\label{sqr}
\end{lemma}

\begin{proof}
Let $g=(a,c)$, $a=gp,c=gq$ with $(q,p)=1$. Then, $b^2=g^2pq$. Therefore, $g|b$, so $b=gr$ for some $r$ and $pq=r^2$. Since $(p,q)=1$, both $p$ and $q$ must be squares. Assume that $p=l^2,q=m^2$ and we find the solutions above.
\end{proof}

Now we get back to the equation.We will concentrate only on primitive solutions of this equation i.e. $(x,y,z)=1$ and use the idea of {\it infinite descent}.
\begin{proof}
We have two cases on based on the parity of $z$.

\paragraph*{\bf Case 1:} $z$ even, so $z=2z_1$ for some $z_1\in\Z$. Then $x$ must be even as well, which gives $x=2x_1$. The equation reduces to $5x_1^2+y^2=z_1^2$ and consider the smallest solution. Assume $(y,z_1)=g$, then $z_1=gz_2,y=gy_1$ and $5x_1^2=g^2(z_2^2-y_1^2)$. If $5|g$, $g=5h$ and thus, $x_1^2=5h^2(z_2^2-y_1^2)$ inferring $5|x$ i.e. $x_1=5x_2$. But then, $5x_2^2=h^2(z_2^2-y_1^2)$ which yields a smaller solution $\left(x_2,\dfrac{y}{5},\dfrac{z_1}{5}\right)$. Therefore, without loss of generality, we can take $(z_1,y)=1$. Now, $5x_1^2=(z_1+y)(z_1-y)$ and $y,z$ both are odd. Write $z_1+y=2A,z_1-y=2B$. From lemma \eqref{cop}, we can say $(A,B)=(z_1,y)=1$ which gives us
\[5x_1^2=4AB\]
Since $(A,B)=1$, $5$ must divide either $A$ or $B$, but not both and $2|x_1$ i.e. $x_1=2x_3$. In the first case, $A=5C$ and $x_3^2=BC$. Using lemma \eqref{sqr}, $B=m^2,C=l^2,x_3=lm$ for some $(l,m)=1$. Thus, $\boxed{z_1=5l^2+m^2,y=5l^2-m^2}$. In the other case, similarly, $\boxed{z_1=l^2+5m^2,y=l^2-5m^2}$.

\paragraph*{\bf Case 2:} This time, $z$ is odd, so $x$ is odd too. And again, we just take $(z,2y)=(z,y)=1$. In a similar fashion to the previous case, $5x^2=(z+2y)(z-2y)$. But $(z+2y,z-2y)=(z,y)=1$ and $z+2y=A,z-2y=B$ with both $A,B\equiv1\pmod2,(A,B)=1$. Therefore, $AB=5x^2$ and it is all over the same one as before. Just a change in the solution set: $\boxed{z=\dfrac{5l^2+m^2}{2},y=\dfrac{5l^2-m^2}{4}}$ or $\boxed{z=\dfrac{l^2+5m^2}{2},y=\dfrac{l^2-5m^2}{4}}$.

\end{proof}

\begin{theorem}
$u^2+uv-v^2$ is a square for infinite pairs of $(u,v)$.
\end{theorem}

\begin{proof}
Let $u^2+uv-v^2=t^2$. Then from the formula of quadratic equation:
\[u=\dfrac{-v\pm\sqrt{5v^2+4t^2}}{2}\]
Since $v$ and $5v^2+4t^2$ are of the same parity, it's enough to prove that there are infinite pairs of $(v,t)$ so that $5v^2+4t^2$ is a square. From theorem \eqref{eq2} we already know that's true. Just take any family of solution.
\end{proof}

Now we can prove the theorem of concern.

\begin{theorem}
There are infinite $(u,v)$ which $\{G_n\}_{n\in\N}$ starts with and is an alternating bisquable sequence. More precisely, this is true for both {\it evenly and oddly bisquable sequence}. The evenly case is true for all $a,b$. But the oddly case can be proven easily for $b=-1$.
\end{theorem}

\begin{proof}
Recall equation \eqref{eq1}. If we are looking for evenly bisquable sequence, then we can see that for all even $n=2k$:
\[G_{2k}G_{2k+2}=G_{2k+1}^2+(-b)^{2k}(u^2+uv-v^2)\]
If $u^2+uv-v^2=t^2$, then we get the equation $G_{2k}G_{2k+1}=r^2+s^2$ for some $r,s$ and for all $k$. Therefore, from theorem \eqref{bisqr}, we can say that, both $G_{2k}$ and $G_{2k+2}$ are bisquare for all $k\in\N$. This provides the proof for the case of evenly sequence.

Now, $n=2k-1$ i.e. $n$ is odd. \[G_{2k-1}G_{2k+1}=G_{2k}^2+u^2+uv-v^2\]
And once again, we are back to the same case as before.

\end{proof}

\begin{remark}
The case for the original Fibonacci sequence is $u=0,v=1,b=1$. Then,
\[G_nG_{n+2}=G_{n+1}^2+(-1)^{n+1}\]
And the proof of its being a bisquable sequence follows immediately for all odd $n$.
\end{remark}

\section{Number Of Divisors Of General Fibonacci Number}
In this section, we restrict our concern on finding a lower bound on the number of divisors of $F_n$, $\tau(F_n)$. But to do that, we will assume $u=0,v=1$ which makes the most general Fibonacci sequence the commonly known General Fibonacci sequence. Additionally, $a,b>0$ so it remains a strictly increasing sequence from $n=2$.

Now, we shall concentrate on $\tau(F_n)$. 

We will need more theorems. This one is a famous one due to {\it Carmichael}, \eqref{car}. Though later we will prove the special case needed for our estimation it in an easier way using other known theorems. We call a prime $p$ a {\it primitive divisor} of $F_n$ if $p|F_n$ but $p\not|F_m$ for $0\leq m<n$. Also, we assume that $\Omega(n)$ is the total number of prime factors(distinct or indistinct) of $n$ i.e. if $n=\prod\limits_{i=1}^kp_i^{e_i}$, 
\[\Omega(n)=\sum_{i=1}^ke_i\]

\begin{theorem}
If $n\neq1,2,6$, then $F_n$ has at least one primitive divisor except when $n=12,a=2,b=1$ and $n=12,a=1,b=1$.
\end{theorem}

From this theorem, we can provide the following estimation:

\begin{theorem}
If $p$ is an odd prime number, then $\tau\left(F_{p^e}\right)\geq2^e$.
\end{theorem}

\begin{proof}
Let $p_1$ be a prime factor of $F_p$. And for $i>1$, we have that $F_{p^i}$ has a primitive factor with respect to $F_{p^i}$. Therefore, if $p_i$ is a primitive factor of $F_{p^i}$(at least one such $p_i$ exist), then for some positive integer $K$,
\begin{eqnarray}
F_{p^e}&=&\prod_{i=1}^ep_iK\label{prp}
\end{eqnarray}
This holds because for every $i$, $F_{p^i}|F_{p^{i+1}}$. We can see that the minimum number of divisor of $\tau(F_{p^e})$ is attained when $F_p=p_1$ and $F_{p_{i+1}}=p_ip_{i+1}$(though we don't study if that's possible here). And in that case, we have, 
\[\tau(F_{p^e})\geq\prod_{i=1}^e(1+1)=2^e\]

\end{proof}

For $p=2$, it will be(since $F_2=1$):

\begin{theorem}
If $e>1$, then $\tau(F_{2^e})\geq2^{e-1}$.
\end{theorem}

Thus, we get an estimation for any positive integer $n>1$. 

\begin{theorem}
\[
\tau(F_n) \geq
\begin{cases}
2^{\Omega(n)}\mbox{ if } n\equiv1\pmod2\\
2^{\Omega(n)}\mbox{ otherwise}
\end{cases}
\]\label{es2}
\end{theorem}

\begin{proof}
First we see that, for $n=\prod\limits_{p^e||n}p^e$,
\[\tau(F_n)\geq\prod_{p^e||n}\tau(F_{p^e})\]
because of equation \eqref{fibcp}.

For odd $n=p_1^{e_1}\cdots p_k^{e_k}$, 

\begin{eqnarray*}
\tau(F_n)&\geq&\prod_{i=1}^k\tau(F_{p_i^{e_i}})\\
		 &\geq&\prod_{i=1}^k2^{e_i}\\
		 & = & 2^{\sum_{i=1}^ke_i}\\
		 & = & 2^{\Omega(n)}
\end{eqnarray*}

For even $n=2^ap_1^{e_1}\cdots p_k^{e_k}$ with $a\geq1$, 

\begin{eqnarray*}
\tau(F_n)&\geq&\tau(F_{2^a})\prod_{i=1}^k\tau(F_{p_i^{e_i}})\\
		 &\geq&2^{a-1}\prod_{i=1}^k2^{e_i}\\
		 & = & 2^{a-1+\sum_{i=1}^ke_i}\\
		 & = & 2^{\Omega(n)-1}
\end{eqnarray*}

\end{proof}

The following theorem improves the current lower bound on $\tau(n)$.

\begin{theorem}
For all $n$, 
\[
\tau(F_n)\geq
\begin{cases}
\tau(n)\mbox{ if }n\equiv1\pmod2\\
\tau(n)-1\mbox{ otherwise}
\end{cases}
\]\label{es1}
\end{theorem}

\begin{proof}
First note that, for all odd $n$ and $d|n$, $d$ is odd. Write down the divisors of $n$ as $d_1,d_2,...,d_{\tau(n)}$ in increasing order. Then from \eqref{fibdiv}, we have for all $i$, $F_{d_i}|F_n$. Since $F_{d_{i+1}}>F_{d_i}$ for odd $d_{i+1}>d_i$. We immediately get that, for every distinct divisor $d$ of $n$, there is a distinct divisor of $F_n$ as well. The proof for the even is analogous to the previous one, only the difference is $F_2$ can be $1$. From $d_2=1$, the sequence is strictly increasing. Hence, proven.

\end{proof}

\begin{remark}
The lower bound we estimated in theorem \eqref{es1} is a lot better than the one in theorem \eqref{es2}. Because 
\[2^e<p^e\]
for all odd $p$, hence $2^{\Omega(n)}\ll\prod_{i=1}^kp_i^{e_i}$ for sufficiently large primes $p_i$.

Also, $F_2=1$ can hold only for $F_n=F_{n-1}+F_{n-2}$, so except for this case, we can in general say,
\[\tau(F_n)\geq\tau(n)\]
\end{remark}

\end{document}